\documentclass[letterpaper, 10 pt, conference]{_Aux/ieeeconf} 
\IEEEoverridecommandlockouts
\let\labelindent\relax
\usepackage{enumitem}

\usepackage{amsthm}
\usepackage{amssymb,amsfonts,mathtools}
\usepackage{graphicx}
\usepackage{cite}
\usepackage{hyperref}
\usepackage{lipsum}
\usepackage[dvipsnames]{xcolor} 
\usepackage{etoolbox}           
\usepackage[normalem]{ulem}     

\newcommand{\SetMode}[1]{%
  \def\CurrentMode{#1}%
  \colorlet{Proof}{blue}
  \colorlet{Added}{red}
  \colorlet{Todo}{magenta}
  \ifstrequal{#1}{review}{%
    \def\ShowComments{1}
  }{}%
  \ifstrequal{#1}{final}{%
    \colorlet{Proof}{black}%
    \colorlet{Added}{black}%
    \colorlet{Todo}{black}%
    \def\ShowComments{0}
  }{}%
}

\newtheorem{definition}{Definition}
\newtheorem{assumption}{Assumption}
\newtheorem{lemma}{Lemma}
\newtheorem{theorem}{Theorem}

\newtheorem{remark}{Remark}
\newtheorem{corollary}{Corollary}

\newcommand{\R}{\mathbb{R}}
\newcommand{\Ub}{\mathbb{U}}
\newcommand{\Sc}{\mathcal{S}}
\newcommand{\Cc}{\mathcal{C}}
\newcommand{\Ac}{\mathcal{A}}
\newcommand{\Bc}{\mathcal{B}}
\newcommand{\Pc}{\mathcal{P}}
\newcommand{\Sb}{\mathcal{S}_{\rm b}}
\newcommand{\kb}{k_{\rm b}}
\newcommand{\hb}{h_{\rm b}}
\newcommand{\Fb}{F_{\rm b}}
\newcommand{\IbN}{\mathbb{I}_N}
\newcommand{\classK}{\mathcal{K}}
\newcommand{\classKK}{\mathcal{K}_\infty \mathcal{K}}

\newcommand{\classKinf}{\mathcal{K}_{\infty}}

\newcommand{\Added}[1]{\textcolor{Added}{#1}}      
\newcommand{\Deleted}[1]{%
  \ifnum\ShowComments=1%
    \textcolor{gray}{\sout{#1}}%
  \fi}
\newcommand{\Todo}[1]{%
  \ifnum\ShowComments=1%
    \textcolor{Todo}{[\textbf{TODO}: #1]}%
  \fi}
\usepackage{xparse}
\NewDocumentEnvironment{roleblock}{m}
  {\begingroup\color{#1}}
  {\endgroup}

\SetMode{final}

\title{\LARGE \bf
Uniform Feasibility For Smoothed Backup Control Barrier Functions
}

\author{
Anil~Alan$^1$,~Bart~De~Schutter$^1$ %
\thanks{
This research has received funding from the European Research Council (ERC) under the European Union's Horizon 2020 research and innovation programme (Grant agreement No. 101018826 - ERC Advanced Grant CLariNet).
$^1$A. Alan and $^1$B. De Schutter are with Delft Center for Systems and Control, Delft University of Technology, 2628 CD Delft, The Netherlands 
{\tt\small \{a.alan, b.deschutter\}@tudelft.nl}.
}%
}%

\begin{document}
\maketitle
\thispagestyle{empty}
\pagestyle{empty}

\begin{abstract}

We study feasibility guarantees for safety filters developed using Control Barrier Functions (CBFs) when a safe set is defined using the pointwise minimum of continuously differentiable functions, a construction that is common for the \emph{backup CBF (BCBF)} method and typically nonsmooth. We replace the minimum by its log--sum--exp (\emph{soft-min}) smoothing and show that, under a strict safety condition, the smooth function becomes a CBF (or extended CBF) for a range of the smoothing parameter. For compact safe sets, we derive an explicit lower bound on the smoothing parameter that makes the smooth function a CBF and hence renders the corresponding safety constraint feasible. For unbounded sets, we introduce tail conditions under which the smooth function satisfies an extended CBF condition uniformly. 
Finally, we apply these results to BCBFs. We show that safety of a compact (terminal) backup set under a backup controller, together with a condition ensuring safety of the backup trajectories on the relevant boundary of the safe set, is sufficient for constraint feasibility for BCBFs.
These results provide a recipe for a priori feasibility guarantees for smooth inner approximations of nonsmooth safe sets without the need for additional online certification.

\end{abstract}


\section{Introduction}
In safety-critical control, the objective is to provide a formal guarantee that, starting from a designated set, trajectories never leave this set during operation. Control Barrier Functions (CBFs) embed such guarantees as constraints in online optimization---often by penalizing the deviation from a desired control law, which leads to a design paradigm called safety filter \cite{ames2019control}. The practicality of this framework ultimately hinges on the feasibility of the resulting optimization problem, which in turn requires a certificate that the constraint function is a valid CBF.

Converse safety theorems study conditions such that a safe set admits a CBF, therefore they can be used to address the certification problem when the safe set is represented by a continuously differentiable function.
Converse results mostly rely on compactness assumption of the set, and span stability-like conditions \cite{prajna2005ontheNecessity}, separation certificate via Meyer functions \cite{wisniewski2016converse}, and robust variants \cite{ratschan2018converse}. Alternatives avoid compactness assumption at the expense of other restrictive conditions such as global Lipschitz constant \cite{liu2022converse} or bounded Lie derivatives \cite{gurriet2018towards}. Recently, a relaxed formulation of CBF, named extended CBF (eCBF), has been proposed to remove these assumptions from the converse side of necessity (of the existence of an eCBF), while retaining safety guarantees on the direct side of sufficiency \cite{mestres2024converse}.

While converse safety theorems enable CBF certification, they presuppose the existence of a control law that keeps the system within the set.
The \emph{backup control barrier function} (BCBF) framework provides a scalable {\color{black} method to construct such a set} \cite{gurriet2018online,chen2021backup,wabersich2021predictive,breeden2022predictive}. 
Here, a backup controller is used to propagate states forward into a smaller safe “backup (terminal) set,” and the set of all states that reach this backup set (at a terminal time) without leaving the safe region is implicitly defined via a minimum over time-indexed constraints, which leads to a nonsmooth.
Identifying a suitably small backup sets can be achieved {\color{black} offline} through Lyapunov equations, convex relaxations or semidefinite programming, e.g., via LMIs \cite{korda2013inner}.
The paradigm has been extended to data-driven settings via Koopman operators \cite{folkestad2020data}, to reduced-order models \cite{molnar2023safety}, to robust formulations \cite{van2024disturbance}, and to the multiple backup controllers setting \cite{janwani2024learning}. 
However, feasibility guarantees in these BCBF-based methods often rely on ad hoc assumptions—such as presuming the existence of suitable class-$\mathcal{K}$ bounds without constructive certificates \cite{gurriet2018online,chen2021backup}, or imposing global compactness \cite{gurriet2020scalable}, or bounded Lie derivatives \cite{rivera2024forward}—which limits systematic deployment.


\textbf{Contributions.} We bridge BCBFs with converse safety certification to find feasibility guarantees for CBF-based safety filters. 
We do this by smoothing the nonsmooth pointwise minimum via the soft-min (log--sum--exp) operator and provide a range of parameters such that converse safety theorems in \cite{mestres2024converse} can be employed for CBF/eCBF certification. In particular:
(i) For compact safe sets, uniform continuity can be invoked to provide CBF conditions, {\color{black} together with a formula to construct a class-$\mathcal{K}$ function guaranteeing the feasibility of the safety constraint}.
(ii) For noncompact sets, we introduce explicit tail conditions under which the soft-min uniformly satisfies an eCBF condition.
(iii) In the BCBF setting, we show that safety of a compact backup set under a backup controller, combined with a condition enforcing safety along the relevant boundary of the safe set, is sufficient to certify a CBF. This relaxes the global compactness assumption in \cite{gurriet2020scalable}.
While the smoothing idea also appeared in \cite{rabiee2025soft} for BCBFs, the approach of \cite{rabiee2025soft} requires computing an inner feasible subset online via an optimization problem for given design parameters. In contrast, our result provides an \emph{a priori} and \emph{uniform} feasibility guarantee that eliminates the need for additional online certification. {\color{black} A numerical example illustrates the resulting design procedure end-to-end.}

\textbf{Organization.} Section~\ref{sec:background} reviews notation, (e)CBFs, converse safety, and the BCBF setup. Section~\ref{sec:main} presents the soft-min construction and uniform thresholds for compact and noncompact cases. Section~\ref{sec:backup} implements these results for BCBFs. {\color{black} Section~\ref{sec:example} demonstrates the approach on a numerical example.} Section~\ref{sec:conclusion} concludes the paper.

\section{Background}
\label{sec:background}

\subsection{Notation}

We write $\IbN \triangleq \{1,\cdots,N\}$ for an integer $N \geq 1$.
The set of continuously differentiable functions is denoted by $C^1$.
For $h : \R^n \to \R$ and $h \in C^1$, its gradient is $\nabla h : \R^n \to \R^n$, \Added{and the transpose of its gradient is denoted as $\nabla^\top h$}.
A value $y \in \R$ is a \emph{regular value} of $h$ if $\nabla h(x) \neq 0$ for every $x$ with $h(x)=y$.
For a vector field $f : \R^n \to \R^n$, $L_f h (x) \triangleq \nabla^\top \! h (x) f(x)$ denotes the Lie derivative of $h$ with respect to (w.r.t.) $f$.
For a set $\Sc\subseteq \R^n$, $\partial \Sc $ and ${\rm Int}(\Sc)$ denote its boundary and interior.
A continuous function $\alpha : [0, a) \to \R_{\ge 0}$ belongs to class-$\mathcal{K}$ if it is strictly increasing and $\alpha(0) = 0$.  
If $a=\infty$ and $\alpha(r) \to \infty$ as $r \to \infty$, then $\alpha \in \classKinf$.
A continuous function $\gamma: \R_{\ge 0} \times \R_{\ge 0} \to \R_{\ge 0}$ belongs to class-$\classKK$ if $\gamma(\cdot,s) \in \classKinf$ for all $s \ge 0$ and $\gamma(r,\cdot) \in \classK$ for all $r\ge0$.

\subsection{Control Barrier Functions and Converse Safety Theorems}
Consider a system
\begin{equation}
\label{eq:system}
    \dot{x}=f(x,u), \quad x(0)=x_0,
\end{equation}
with $x\in\R^n$ and $u\in\Ub\subseteq\R^m$, where $f:\R^n\times\R^m\to\R^n$ is locally Lipschitz.
We aim at \emph{safety}, that the closed-loop system with a state-feedback controller $u=k(x)$ stays in a predefined set $\Sc$.
We will consider sets defined as  
\begin{equation}
\label{eq:Sc}
\begin{split}
    \Sc \triangleq \{ x\in\R^n ~|~ h(x)\geq0 \}, \quad h:\R^n\to\R, ~~ h \in C^1.
\end{split}
\end{equation}

Control Barrier Functions (CBFs) can be used to construct controllers with a safety guarantee.
\begin{definition}[CBF, \cite{ames2019control}]
\label{def:CBF}
A function $h\in C^1$ is a \emph{CBF} for the system \eqref{eq:system} if there exists a function $\alpha \in \classK$ satisfying:
\begin{equation}
  \sup_{u\in \Ub} \Big[ \underbrace{ \nabla^\top \! h(x) f(x,u)}_{L_fh(x,u)} \Big] \ge -\alpha(h(x)), \quad \forall x \in \Sc,
  \label{eq:cbf_ineq}
\end{equation}
where $L_fh$ denotes the Lie derivative of $h$ w.r.t. $f$.
\end{definition}

We will utilize a less restrictive version of CBF called \emph{extended CBF} to obtain more general results.
\begin{definition}[{Extended CBF, \cite{mestres2024converse}}]
\label{def:eCBF}
    A function $h \in C^1$ is an \emph{extended CBF} for \eqref{eq:system} if there exists a function $\gamma \in \classKK$:
    \begin{equation}
    \label{eq:eCBF}
        \sup_{u\in \Ub} \left[ L_fh(x,u) \right] \ge -\gamma\left( h(x), \|x\| \right) , \quad \forall x \in \Sc.
    \end{equation}
\end{definition}

CBFs and eCBFs imply that any locally Lipschitz continuous controller satisfying \eqref{eq:cbf_ineq} or \eqref{eq:eCBF} ensures safety, 
{ provided that 0 is a regular value of $h$} \cite{ames2019control,mestres2024converse}. 
One such controller that safeguards a desired (locally Lipschitz continuous) controller $k_{\rm des}$ is the \emph{safety filter}:
\begin{equation}
\begin{aligned}
k_{\mathrm{SF}}(x) = \underset{u \in \Ub}{\rm argmin}~\ & {\color{black} \|u - k_{\mathrm{des}}(x)\|_2^2} \\
\text{s.t.}\ & L_f h(x,u) \geq -\alpha(h(x)),
\end{aligned}
\label{eq:safetyfilter}
\end{equation}
where $\alpha(h(x))$ can be replaced with $\gamma(h(x),\|x\|)$ for eCBF, {\color{black} and $\|\cdot\|_2$ denotes the Euclidean norm}. 
While certification of $h$ as a CBF or an eCBF implies the feasibility of the safety constraint, the optimization problem \eqref{eq:safetyfilter} becomes a QP (with efficient real-time solvers available) if $f$ is affine-in-control and $\Ub$ can be represented with affine constraints.
Converse safety theorems certify $h$ as a CBF or an eCBF.
\begin{theorem}[Converse safety for compact sets, \cite{ames2019control}]
\label{theo:converse_compact}
    Let $\Sc$ be a compact set defined with $h\in C^1$ and let 0 be a regular value of $h$. If there exists a controller $k:\R^n\to\Ub$ such that $F(x)=f(x,k(x))$ is safe w.r.t. $\Sc$, then $h$ is a CBF.
\end{theorem}

\begin{theorem}[Converse safety for general case, \cite{mestres2024converse}]
\label{theo:converse}
    Let $\Sc$ be a (potentially unbounded) set defined with $h\in C^1$ and let 0 be a regular value of $h$. 
    If there exists a controller $k:\R^n\to\Ub$ such that the closed loop system $F$ is safe w.r.t. $\Sc$, then $h$ is an eCBF.
\end{theorem}

\begin{remark}
    The benefit of eCBFs with a more general class of functions $\gamma\in \classKK$ over $\alpha\in\classK$ make them more suitable for unbounded sets.
    For example, even when $L_Fh(x)$ becomes arbitrarily negative as $\|x\|\to\infty$, it is still plausible to construct a lower bound $\gamma (\cdot,\|x\|)$ using $L_Fh$ thanks to the second term.
    However, missing this term, a CBF candidate may fail to admit a function $\alpha\in \classK$, see Example IV.1 in \cite{mestres2024converse}. 
\end{remark}

\subsection{Backup Control Barrier Functions (BCBF)}
\label{subsec:backup}
While converse safety theorems give us strong theoretical guarantees to certify \emph{any} $h$ defining $\Sc$ as a CBF or an eCBF, they essentially require the existence of a control law rendering $\Sc$ safe. 
A solution is proposed in \cite{gurriet2018online}.
This idea, which will be called \emph{backup control barrier functions (BCBF)} in this paper, relies on constructing a (larger) safe set using finite-time safe trajectories that evolve within a (smaller) safe set $\Sb \subseteq \Sc$. 

Consider $\hb : \R^n \to \R$, $\hb \in C^1$ defining a set $\Sb = \{ x ~|~ \hb(x) \ge 0 \}$, and the closed-loop system with a locally Lipschitz continuous backup controller $\kb : \R^n \to \Ub $:
\begin{equation}
\label{eq:system_closed_backup}
    \dot{x} = \Fb(x) \triangleq f(x,\kb(x)), ~~x(0)=x_0. 
\end{equation}
Solution trajectories to \eqref{eq:system_closed_backup}, denoted as $\phi(x_0,\tau)$, can be used over a finite horizon $\tau \in [0,T]$ with a $T\geq0$ to construct
\begin{equation}
\label{eq:ST}
    \Sc_T \!=\! \big\{ x \in \R^n ~|~ h_T(x) \ge 0  \big\},
\end{equation}
where $
\label{eq:h_T}
    h_T(x) \triangleq \min \left\{ \min_{\tau \in [0,T]} h( \phi(x, \tau) ) , \hb( \phi(x,T) ) \right\}.
$
The set $\Sc_T$ contains all points $x \in \Sc$ such that, starting from $x$ at $\tau = 0$, the system reaches $\Sb$ by $\tau = T$ without leaving $\Sc$ at any $\tau \in [0,T]$ under the backup controller. 
Therefore, if $\Sb$ is a safe set, so is $\Sc_T$, see \cite{gurriet2020scalable,molnar2023safety} for a more formal statement.

The set $\Sc_T$ can potentially be used in converse safety theorems to certify $h_T$ as a CBF or an eCBF, which then guarantees the feasibility of the safety constraint with $h_T$. However, Theorems~\ref{theo:converse_compact}-\ref{theo:converse} require a continuously differentiable constraint function, whereas $h_T$ is not due to the minimum.
To solve this problem, we will use a smooth inner approximation of $\min$ functions in $h_T$ in the next section.
Before that, we use discretization for the minimum in $h_T$ over a continuous time interval $[0,T]$.
In particular, consider a discretization $\tau_i \in \{0, \Delta \tau, \cdots , T\} $ with a uniform interval $\Delta \tau > 0$ and $i \in \IbN$ with $N = T / \Delta \tau + 1$. 
Discretization leads to an approximation of ${\Sc}_T$ defined with
\begin{equation}
\label{eq:bmin}
    \hat \Bc  \triangleq \{x \in \R^n ~|~ \hat b(x)\ge0\}, \quad \hat b(x)  \triangleq \min_{i \in \IbN} b_i(x),
\end{equation} 
where
\begin{align}
    \label{eq:BCFB_bi}
    b_i(x) \triangleq \begin{cases}
    h\left( \phi(x,\tau_i) \right), &~{\rm if}~ i \in \IbN \setminus \{N\}, \\
    \hb \left( \phi(x,T) \right), &~{\rm if}~ i = N.
    \end{cases}
\end{align}
We can replace the safety constraint of \eqref{eq:safetyfilter} with:
\begin{equation}
    \label{eq:BCBF_condition}
    L_fb_i(x,u) \geq -\alpha_i \left( b_i(x) \right), \quad \forall i \in \{1,\cdots,N\}.
\end{equation}
with $\alpha_i \in \classK$ for the CBF case and a $\gamma_i \in \classKK$ for the eCBF case. 
While constraints \eqref{eq:BCBF_condition} are sufficient for safety\footnote{Technically, discretization yields a relaxation of the original safe set. While it is possible to shift the safety guarantee to $\hat{\Bc}$ by tightening  constraints against imperfections emerging from discretization \cite{gurriet2020scalable}, we skip these details here to keep the focus on the merits of the approach presented.} \cite{he2025predictive}, their feasibility depends on whether $\hat b$ can be certified as a CBF or an eCBF.

Our main goal in this paper is to obtain feasibility guarantees for the safety constraint of CBF-based controllers such as \eqref{eq:safetyfilter} when the constraint is constructed as smoothed inner approximation of pointwise minimum of subconstraint functions. 
In the next section we set up this problem for a general setting and provide solutions for two separate cases: first, when the set $\hat \Bc$ is compact (using Theorem~\ref{theo:converse_compact}); then for a general case (using Theorem~\ref{theo:converse}). 
Finally, we will investigate how these conditions translate back to the BCBF framework in Section~\ref{sec:backup}.

\section{Main Results}
\label{sec:main}

\subsection{Problem Formulation}
\label{subsec:problem}

Let a set $\hat \Sc$ be defined as
\begin{equation}
\label{eq:Shat}
    \hat \Sc=\{x\in\R^n ~|~ \hat h(x) \ge 0\}, \quad \hat h(x) \triangleq \min_{i \in \IbN} h_i(x),  
\end{equation}
with functions $h_i:\R^n\to\R$ and $h_i \in C^1$ for all $i \in \IbN$, and an integer $N \ge 1$.
The \emph{active} and \emph{inactive set} of indices at a point $x \in \hat \Sc$ are denoted with  $\Ac(x) \triangleq \{i \in \IbN ~|~ h_i(x)=\hat h(x)\}$ and $\Pc(x) \triangleq  \IbN \setminus \Ac(x)$. 
%
We assume that the Mangasarian-Fromovitz constraint qualification (MFCQ) holds for $\partial \hat \Sc$, i.e., there exists a $v\in\R^n$ such that
\begin{equation}
\label{eq:MFCQ}
\nabla^\top \! h_i(x) \,v > 0, \quad \forall i\in \mathcal A(x), ~~\forall x\in\partial \hat \Sc.
\end{equation}
Note that this assumption is a geometric regularity condition on $\partial \hat \Sc$, and it is a property of the set. It guarantees that, at all points on the boundary, the active constraint surfaces meet in a non-degenerate way, which then ensures that the contingent tangent cone\footnote{%
{
$T_{\hat \Sc}(x)$ denotes the contingent cone of a set $\hat \Sc \subseteq \R^n$ at a point $x\in\R^n$, i.e., the set of all directions $v\in\R^n$ for which there exist sequences \Added{for $k\to\infty$}, $t_k\downarrow 0$ and $x_k\in\hat{\mathcal S}$ with 
$x_k\to x$ and $(x_k - x)/t_k \to v$.}%
} 
of $\hat \Sc$ at a point $x\in \partial \hat \Sc$ is full-dimensional and has the representation
\begin{equation}
T_{\hat \Sc}(x) = \{ v\in\R^n ~|~ \nabla^\top \! h_i(x) \,v \ge 0, ~~\forall i\in\Ac(x)\}.
\end{equation}
When $h_i \in C^1, ~\forall i \in \IbN$, MFCQ is satisfied whenever all the active gradients $\nabla h_i$ do not point in exactly opposite directions, so a common direction can increase all active constraints simultaneously. 

We are interested in a case where a locally Lipschitz continuous controller $k:\R^n\to\Ub$ endows the system 
\begin{equation}
\label{eq:system_closed}
    \dot{x}=F(x) = f(x,k(x)),  \quad x(0)=x_0,
\end{equation}
with a safety guarantee w.r.t. $\hat{\Sc}$.
Before introducing the problem setting, we formally define safety. 
\begin{definition}[Strict safety]
\label{def:safety}
    A system \eqref{eq:system_closed} is strictly safe w.r.t. a set $\hat \Sc$ if
    \begin{equation}
    \label{eq:safety}
    \forall x_0 \in \hat\Sc \implies x(t) \in {\rm Int}(\hat\Sc), ~\forall t>0.
    \end{equation}
\end{definition}
It is noted that the safety definition in Definition~\ref{def:safety} is a stricter notion than the safety definition commonly used within the CBF framework \cite{ames2019control}.
In particular, { inspired by the \emph{strict invariance} in \cite{aubin1991viability}, we consider cases where any trajectory on the boundary cannot stay on the boundary, and it is strictly pushed inside.
Defined on a stricter condition, Theorems~\ref{theo:converse_compact}-\ref{theo:converse} are still applicable for the safety definition in Definition~\ref{def:safety}.}

According to the strict invariance result in Aubin~\cite[$\S4.3$]{aubin1991viability}, 
$F(x)$ must belong to the \emph{relative interior} of the contingent cone $T_{\hat \Sc}(x)$ for all $x \in \partial \hat \Sc$ to ensure strict safety.
When $\hat \Sc$ is defined as in \eqref{eq:Shat}, and under MFCQ, the strict invariance condition becomes equivalent to \cite{clarke1998nonsmooth}
\begin{equation}
\label{eq:safetyassumption}
    L_F h_i(x) > 0, 
    \quad \forall i \in \mathcal A(x),\ \forall x \in \partial \hat{\mathcal S}.
\end{equation}

\noindent {\bf Soft-min:} To employ converse safety theorems, we construct a smooth inner approximation of $\hat \Sc$ using the soft-min (or log-sum-exp minimum) function:
\begin{equation}
    \label{eq:Stilde}
    \tilde{\Sc}_\theta \!=\! \{ x\in\R^n ~|~ \tilde{h}_\theta (x) \!\ge\! 0 \}, ~\tilde{h}_\theta (x) \!\triangleq\! -\frac{1}{\theta} \log \left( \sum_{i=1}^N {\rm e}^{-\theta h_i(x)} \right),
\end{equation}
where $\theta > 0$ is a parameter adjusting the transition between two extreme cases: $\theta \to 0$ yielding equal contribution from all $h_i$, and $\theta \to \infty$ returning to only active constraints.  
Since $h_i \in C^1$ for all $i \in \IbN$, we can find the gradient:
\begin{equation}
    \label{eq:grad_htilde}
    \nabla \tilde h_\theta (x) = \sum_{i=1}^N w_{\theta,i}(x) \nabla h_i (x), ~w_{\theta,i}(x) \triangleq \frac{ {\rm e}^{-\theta h_i(x)} }{ \sum_{j=1}^N {\rm e}^{-\theta h_j(x)} },
\end{equation}
where $w_{\theta,i}:\R^n \to (0,1)$ are weights of how much each gradient $\nabla h_i$ contributes to $\nabla \tilde h_\theta$. 
Note $\sum_{i \in \IbN} w_{\theta,i}(x) = 1$ for all $x \in \hat\Sc$. 

\noindent {\bf Problem:} Given that \eqref{eq:safetyassumption} holds {(implied by strict safety w.r.t. $\hat \Sc$ and MFCQ),} our particular goal is to show that there exists a $\theta$ such that 
we can use Theorems~\ref{theo:converse_compact}-\ref{theo:converse} to certify $\tilde h_\theta$ in \eqref{eq:Stilde} as a CBF or an eCBF.
More precisely, we look for a uniform $\theta^* > 0$ such that the following holds
\begin{equation}
\label{eq:maingoal}
    \nabla^\top \! \tilde h_\theta(x) F(x) = L_F \tilde h_\theta(x) > 0, \quad \forall x \in \partial \tilde \Sc_\theta, ~\forall \theta > \theta^*.
\end{equation}
The strict positiveness of the Lie derivative $L_F \tilde h_\theta$ in \eqref{eq:maingoal} implies that $\nabla \tilde h_\theta(x) \!\neq\! 0$ for all 
$x\in\partial \tilde \Sc_\theta$; so $0$ is a regular value of $\tilde h_\theta$, and $\partial \tilde \Sc_\theta$ is 
a $C^1$ hypersurface. Therefore, \eqref{eq:maingoal} becomes sufficient for Aubin’s strict invariance condition \cite{aubin1991viability}, and Theorems~\ref{theo:converse_compact}-\ref{theo:converse} apply.

We address the problem of obtaining \eqref{eq:maingoal} by considering two cases separately. First, when $\hat{\mathcal{S}}$ is compact, we establish the existence of a $\theta^*$ without requiring any additional assumptions. Subsequently, we examine the case in which $\hat{\mathcal{S}}$ is unbounded and show that a uniform $\theta^*$ can still be obtained, provided that additional conditions on the subconstraint functions $h_i$ and the closed-loop vector field $F$ are satisfied.

\begin{remark}
\label{rem:conservativeness}
{\color{black}
The soft-min smoothing yields an inner approximation $\tilde \Sc_\theta \subseteq \hat \Sc$, and the gap can be quantified by the two-sided bound $\hat h(x) - \frac{\log N}{\theta} \le \tilde h_\theta(x) \le \hat h(x), \quad \forall x \in \hat \Sc$.
So, the loss in the certified safe set vanishes as $\theta \to \infty$. While this conservativeness is the price paid for replacing the nonsmooth $\hat h$ with a $C^1$ function, it is precisely this smoothness that enables the feasibility guarantees developed in the next section. 
}
\end{remark}

\begin{remark}
    While both our approach and that of \cite{rabiee2025soft} employ the same type of smooth approximation \eqref{eq:Stilde} for the minimum operator \eqref{eq:Shat}, the nature of the resulting feasibility guarantees differs.
    In \cite{rabiee2025soft}, feasibility of the safety filter is only ensured on a subset $\Gamma \subseteq \tilde \Sc_\theta$, where $\Gamma$ is determined through an auxiliary optimization problem.
    In our work, we instead derive conditions under which feasibility holds on $\tilde \Sc_\theta$, and these conditions are obtained prior to implementation, without the need for additional online certification.
\end{remark}

\subsection{Compact Sets}
\label{sec:compact}

In this part, we will consider the case that $\hat \Sc$ is a compact set. We start by defining an $\varepsilon$-tube around $\partial \hat \Sc$ as
\begin{equation}
    \label{eq:tube}
    T_\varepsilon = \{ x \in \hat{\Sc} ~|~ 0 \le \hat h(x) \le \varepsilon \}.
\end{equation}
Note that, for $\varepsilon>0$, we have $T_\varepsilon \subseteq \hat \Sc$.
A property of soft-min functions is a lower bound: 
\begin{equation}
\label{eq:softmin_bounds}
    \hat h(x)-\frac{\log N}{\theta} \le \tilde h_\theta(x), \quad \forall x \in \hat \Sc, ~\forall \theta > 0.
\end{equation}
The lower bound in \eqref{eq:softmin_bounds} implies the existence of a $\theta$ such that the boundary of $\tilde \Sc_\theta$ is \emph{fully contained} in the $\varepsilon$-tube around $\partial \hat \Sc$; to be precise, $\partial \tilde \Sc_\theta \subset T_\varepsilon \subseteq \hat \Sc$ for all $\theta \ge \tfrac{\log N}{\varepsilon}$.
Therefore, given a $\theta \ge \tfrac{\log N}{\varepsilon}$, showing that $L_F \tilde h_\theta(x)>0$ holds for all $x\in T_\varepsilon$ is \emph{sufficient} for \eqref{eq:maingoal}.
Furthermore, since $\hat \Sc$ is compact, so are $T_\varepsilon$ and $\partial \tilde \Sc_\theta$ for an $\varepsilon>0$ and a $\theta > 0$.

We can decompose the Lie derivative in \eqref{eq:maingoal}:
\begin{equation}
\label{eq:decomposition}
\begin{alignedat}{1}
    L_F \tilde h_\theta(x) \!=\! \sum_{\!\!\!i \in \Ac(x)\!\!} w_{\theta,i}(x) L_Fh_i(x) \!+\! \sum_{\!\!\!j \in \Pc(x)\!\!} w_{\theta,j}(x) L_Fh_j(x),
\end{alignedat}
\end{equation}
where $w_{\theta,(\cdot)}$ are the weights given as in \eqref{eq:grad_htilde}, characterizing how much the Lie derivative of each subconstraint contributes to the total Lie derivative $L_F \tilde h_\theta$.
Notice that the summation is separated between active and inactive constraints, cf. \eqref{eq:grad_htilde}, as this will be useful to check $L_F \tilde h_\theta(x)>0$ for all $x\in T_\varepsilon$. 
In particular, in order for the strict positiveness to hold, we need certain bounds on the terms in \eqref{eq:decomposition}, which are presented in the next lemma. 
Before presenting this, consider the following notation defining the gap between a subconstraint function $h_j$ and the pointwise minimum: $\delta_j(x) \triangleq h_j(x) - \hat h(x)$. Notice that by definition we have $\delta_i(x) = 0$ for all active constraints $i \in \Ac(x)$, whereas $\delta_j(x) > 0$ for inactive ones $j \in \Pc(x) $.

\begin{lemma}
\label{lem:compact}
Let $\hat \Sc$ defined as in \eqref{eq:Shat} be compact with $h_i \in C^1, ~\forall i \in \IbN$ with $N \ge 1$, {and assume MFCQ holds for $\partial \hat \Sc$}. If the system \eqref{eq:system_closed} is {strictly} safe w.r.t. $\hat \Sc$, then there exists a $T_\varepsilon\subseteq \hat \Sc$ with an $\varepsilon > 0$ such that following conditions hold $\forall x \in T_\varepsilon$ with bounds $M,r,d>0$ 
\begin{align}
\label{eq:lemma_compact_1}
    \max_{i \in \IbN}  \left| L_Fh_i(x) \right|  &\leq M, \\
\label{eq:lemma_compact_2}
    \min_{i \in \Ac(x)}  L_Fh_i(x) &\geq r > 0, \\
\label{eq:lemma_compact_3}
    \min_{j \in \Pc(x)} \delta_j(x) &\geq d > 0,
\end{align}
\end{lemma}

\begin{proof}
\noindent\textbf{1) Proof of \eqref{eq:lemma_compact_1}.}
Notice that $T_\varepsilon \subseteq \hat \Sc$ is compact as $\hat \Sc$ is compact.
As $h_i \in C^1$ and $F$ is continuous, the map $x\mapsto L_F h_i(x)$ is continuous for each $i\in\mathbb I_N$, and therefore bounded on the compact set $T_\varepsilon$. Then \eqref{eq:lemma_compact_1} follows with
\begin{equation}
M \triangleq \max_{i\in\mathbb I_N}\; \max_{x\in T_\varepsilon} |L_F h_i(x)|.
\end{equation}

\smallskip

\noindent\textbf{2) Proof of \eqref{eq:lemma_compact_2}.}
Define a set
\begin{equation}
\Gamma \triangleq \{\,(x,i) ~|~ x\in \partial\hat \Sc,\ i \in \mathcal A(x)\}.
\end{equation}
The strict safety condition and MFCQ yields the strict positiveness of the Lie derivative \eqref{eq:safetyassumption}, i.e., $L_F h_i(x) > 0, \forall (x,i) \in \Gamma$.
Since $\partial\hat \Sc$ is compact and $\mathcal A(x)\subseteq \mathbb I_N$ is finite for every $x \in \partial\hat \Sc$, $\Gamma$ is compact.  
Because each $L_F h_i$ is continuous, the map $(x,i)\mapsto L_F h_i(x)$ is continuous on $\Gamma$ and therefore attains a strictly positive minimum:
\begin{equation}
r_0 \triangleq \min_{(x,i)\in\Gamma} L_F h_i(x) > 0,
\end{equation}
By the continuity of the map $L_F h_i$, we may choose an $\varepsilon>0$ sufficiently small such that 
\begin{equation}
L_F h_i(x) \ge r_0/2 > 0, \quad i\in \mathcal A(x), ~x \in T_\varepsilon.
\end{equation}
Then \eqref{eq:lemma_compact_2} follows with $r = r_0/2$.

\smallskip

\noindent\textbf{3) Proof of \eqref{eq:lemma_compact_3}.}
Define a set
\begin{equation}
\Lambda \triangleq \{\,(x,j) ~|~  x\in \partial\hat \Sc, ~j\in \mathcal P(x)\,\}.
\end{equation}
Similar to $\Gamma$ as in the previous case, $\Lambda$ is a compact set.  
The map $(x,j)\mapsto \delta_j(x)$ is continuous and strictly positive on $\Lambda$ (by definition), so it attains a minimum:
\begin{equation}
d_0 \triangleq \min_{(x,j)\in\Lambda} \delta_j(x) > 0.
\end{equation}
By the continuity of the map $\delta_j$, we may choose an $\varepsilon>0$ sufficiently small such that
\begin{equation}
\delta_j(x) \ge d_0/2 > 0, \quad j \in \Pc(x), ~x\in T_\varepsilon.
\end{equation}
Then \eqref{eq:lemma_compact_3} follows with $d = d_0/2$.
\end{proof}

Using the bounds from Lemma~\ref{lem:compact}, we can now specify a range of $\theta$ such that $L_F \tilde h_\theta(x)>0$ holds for all $x\in T_\varepsilon$,
which implies our goal of \eqref{eq:maingoal}.

\begin{theorem}[$\tilde h_\theta$ is a CBF for compact sets]
\label{theo:compact}
Let the assumptions of Lemma~\ref{lem:compact} hold with $\varepsilon,M,r,d>0$. Then, $\tilde h_\theta$ is a CBF for \eqref{eq:system} for all $\theta > \theta^*$ where
    \begin{equation}
    \label{eq:thetastar_compact}
        \theta^* = \max \left\{ \frac{\log N}{\varepsilon},  \frac{1}{d} \log{ \left( \frac{N(r+M)}{r} \right) }   \right\},
    \end{equation}
    where $N$ is the number of subconstraint functions $h_i$.
\end{theorem}

\begin{proof}
Recall the decomposition of the Lie derivative $L_F \tilde h_\theta$ between active and inactive constraints as in \eqref{eq:decomposition}.
Using the bounds $M,r,d>0$ from Lemma~\ref{lem:compact}, one can obtain 
\begin{equation}
\label{eq:temp1}
    L_F \tilde h_\theta(x) \geq r \sum_{i \in \Ac(x)} w_{\theta,i}(x) - M \sum_{j \in \Pc(x)} w_{\theta,j}(x),~~ \forall x \in T_\varepsilon.
\end{equation}
The gap bound $\delta_j(x) \geq d$ implies (since $\theta > 0$)
\begin{equation}
\label{eq:temp70}
    h_j(x) \geq \hat h(x) + d \implies {\rm e}^{-\theta h_j(x)} \leq {\rm e}^{-\theta (\hat h(x)+d)}.
\end{equation}
Using \eqref{eq:temp70} and \eqref{eq:grad_htilde} one gets
\begin{equation}
\label{eq:temp2}
    \sum_{j \in \Pc(x)} w_{\theta,j}(x) \leq N {\rm e}^{-\theta d}.
\end{equation}
Since $\sum_{i\in\mathcal A(x)}w_{\theta,i}(x) = 1 - \sum_{j\in\mathcal P(x)}w_{\theta,j}(x)$,  \eqref{eq:temp1} and \eqref{eq:temp2} yield 
\begin{equation}
\label{eq:temp3}
    L_F \tilde h_\theta(x) \ge r - (r+M)N e^{-\theta d}, ~~ \forall x \in T_\varepsilon.
\end{equation}
To ensure strict positiveness, we require
\begin{equation}
\label{eq:temp4}
    r - (r+M)N {\rm e}^{-\theta d} > 0 \iff {\rm e}^{-\theta d} < \frac{r}{N(r+M)},
\end{equation}
which holds whenever
\begin{equation}
\label{eq:temp5}
    \theta > \theta^*_{\rm core} \triangleq \frac{1}{d} \log{ \left( \frac{N(r+M)}{r} \right) }.
\end{equation}
Notice that $\theta^*_{\rm core}$ is the second term in \eqref{eq:thetastar_compact}.

The soft-min lower bound \eqref{eq:softmin_bounds} implies that choosing $\theta \geq \tfrac{\log{N}}{\varepsilon}$ ensures $\partial \tilde \Sc_\theta \subset T_\varepsilon$.
Notice $\tfrac{\log{N}}{\varepsilon}$ is the first term in \eqref{eq:thetastar_compact}. Thus, choosing $\theta > \theta^*$ with $ \theta^*$ as in \eqref{eq:thetastar_compact} satisfies \eqref{eq:maingoal}. 

Finally, the strict positiveness of the Lie derivative implies $\nabla \tilde h_\theta(x)\neq 0$ on $\partial\tilde \Sc_\theta$, and thus $0$ is a regular value of $\tilde h_\theta$. Then $\tilde \Sc_\theta$ is {strictly} safe based on Aubin's strict invariance condition \cite{aubin1991viability}; and by Theorem~~\ref{theo:converse_compact}, $\tilde h_\theta$ is a CBF $\forall \theta > \theta^*$.    
\end{proof}

{\color{black}
While Theorem~\ref{theo:converse_compact} certifies $\tilde{h}_\theta$ as a CBF and thereby guarantees the existence of a class-$\mathcal{K}$ function $\alpha$ for which the safety constraint in~\eqref{eq:safetyfilter} is feasible, it does not by itself prescribe such an $\alpha$. 
By leveraging the converse safety theorem, however, one can derive a condition that yields a lower bound for such a class-$\mathcal{K}$ function. In particular, a linear class-$\mathcal{K}$ function $\alpha(r) = \gamma r$ leads to a feasible safety constraint whenever $\gamma$ satisfies
\begin{equation}
\gamma \geq \gamma^* \triangleq \frac{L_{\rm max}}{\varepsilon}, ~ L_{\rm max} \triangleq \displaystyle\sup_{x \in \tilde S_\theta}\Big[\max\bigl(0,-L_F \tilde h_\theta(x)\bigr) \Big]
\label{eq:gamma_star}
\end{equation}
with $\varepsilon$ as in Lemma~\ref{lem:compact}, and $L_{\rm max}$ representing the worst rate of change of $\dot{\tilde h}_\theta$ on $\tilde \Sc_\theta$.
Inside the tube $T_\varepsilon$, the proof of Theorem~\ref{theo:compact} already ensures $L_F \tilde h_\theta(x) > 0$, so the CBF condition holds for any $\gamma > 0$. Outside the tube, the bound~\eqref{eq:gamma_star} ensures that $\gamma\,\tilde h_\theta(x) \geq -L_F \tilde h_\theta(x)$. 
}

\begin{remark}[Computability of $\theta^*$ and $\gamma^*$]
\label{rem:computability}
{\color{black}
All quantities entering the threshold $\theta^*$ in~\eqref{eq:thetastar_compact} and the bound~\eqref{eq:gamma_star} on $\gamma$ can be determined offline from the system data: the constants $\varepsilon$, $M$, $r$, and $d$ in $\theta^*$ follow from Lemma~\ref{lem:compact}, and $L_{\rm max}$ can be evaluated numerically by sampling once $\theta$ is fixed. Together, these provide a recipe for a priori feasibility of the safety constraint in~\eqref{eq:safetyfilter} uniformly over $\tilde \Sc_\theta$. 
}
\end{remark}

\subsection{General Case}
\label{subsec:unbounded}

In this part we study a set of sufficient conditions implying similar results without relying on the compactness property.
Since $T_\varepsilon$ may be unbounded, the uniform bounds \eqref{eq:lemma_compact_1}-\eqref{eq:lemma_compact_3} need not hold and \eqref{eq:safetyassumption} alone is not sufficient for \eqref{eq:maingoal}.
We therefore decompose 
$T_\varepsilon = (T_\varepsilon \cap B_R) \cup (T_\varepsilon \cap B_R^{\rm c})$ with $B_R \triangleq \{x ~|~ \|x\| \le R\}$ and $B_R^{\rm c}$ denotes the complement of $B_R$. The compact \emph{core} (the first term in the set union) inherits the bounds of Section~\ref{sec:compact}, and we impose additional conditions on the unbounded \emph{tail end} (the second term) to control inactive terms in $L_F \tilde h_\theta$.



\begin{assumption}
\label{ass:unbounded}
    There exist constants $\varepsilon,R>0$ such that following statements are true for every $x \in T_\varepsilon \cap B_R^{\rm c}$:
    \begin{enumerate}[label=(\alph*)]
        \item There exists a nondecreasing function $\eta : [R,\infty)\to \R_{>0}$ with $\eta(r)\to\infty$ as $r\to\infty$ that satisfies
        \begin{equation}
              \min_{ j \in \Pc(x)} \delta_j(x) \geq \eta(\|x\|).
        \end{equation}
        \item There exist constants $C,p\ge0$ satisfying 
        \begin{equation}
              \max_{j \in \Pc(x)} |L_Fh_j(x)| \le C ( 1 + \|x\| )^p.
        \end{equation}      
        \item There exists a constant $r_\infty>0$ satisfying  
        \begin{equation}
              \min_{i \in \Ac(x)} L_Fh_i(x) \geq r_\infty.
        \end{equation}   
    \end{enumerate}
\end{assumption}


Assumption~\ref{ass:unbounded} ensures that, on the tail end, the active part of $L_F \tilde h_\theta$ (the first term in \eqref{eq:decomposition}) dominates the inactive part (the second term): (a) forces $w_{\theta,j}\to 0$ exponentially for inactive constraints, which overcomes the polynomial growth allowed by (b).
As a result, the inactive part in \eqref{eq:decomposition} goes to 0 as $\|x\|\to \infty$ or $\theta \to \infty$. At any fixed $x$, we can find a finite $\theta^*$ such that the inactive term in \eqref{eq:decomposition} falls within the strictly positive margins of the uniform bound of the active part (that is, $r_\infty$ as assumed by Assumption~\ref{ass:unbounded}\,(c)).


\begin{remark}
Assumption~\ref{ass:unbounded}\,(a) holds, for example, when there is a unique minimizer $i^*$ of $\hat h$ at the tail end and the inactive gaps are radially coercive ($\delta_j(x)\to\infty$ as $\|x\|\to \infty$); then $\eta(r) \triangleq \min_{j \neq i^*} \min_{\|x\|=r} \left( h_j(x) - h_{i^*}(x) \right)$ satisfies the requirement.
\end{remark}

\begin{theorem}[$\tilde h_\theta$ is an eCBF for general case]
\label{theo:unbounded}
Let $\hat \Sc$ be defined as in \eqref{eq:Shat} with $h_i \in C^1, \forall i \in \IbN$ with $N \ge 1$, {assume MFCQ holds for $\partial \hat \Sc$}, and suppose that Assumption~\ref{ass:unbounded} holds.
If the system \eqref{eq:system_closed} is {strictly} safe w.r.t. $\hat \Sc$, then there exists a $\theta^*>0$ such that $\tilde h_\theta$ is an eCBF for \eqref{eq:system} for all $\theta > \theta^*$.
\end{theorem}

\begin{proof}

We analyze the core and tail regions separately.

\smallskip
\noindent\textbf{Core region $(T_\varepsilon \cap B_R)$.}
The proof of Theorem~\ref{theo:compact} yields a critical value $\theta^*_{\rm core}$ as in \eqref{eq:temp5} with terms $r,M,d>0$ whose existence is guaranteed per Lemma~\ref{lem:compact} as $T_\varepsilon \cap B_R$ is compact for any given $R>0$.

\smallskip
\noindent\textbf{Tail region $(T_\varepsilon \cap B_R^{\rm c})$.}
Recall the decomposition \eqref{eq:decomposition}.
Assumption~\ref{ass:unbounded}\,(b)-(c) give, for all $x\in T_\varepsilon \cap B_R^{\rm c}$,
\begin{equation}
    L_F \tilde h_\theta(x) \geq r_\infty \!\!\sum_{i \in \Ac(x)} w_{\theta,i}(x) - C(1+\|x\|)^p \!\!\!\sum_{j \in \Pc(x)} w_{\theta,j}(x).
\end{equation}

The radially coercive gap $\min_{j\in\mathcal P(x)}\delta_j(x) \ge \eta(\|x\|)$ by Assumption~\ref{ass:unbounded}\,(a) can be used in a similar way to \eqref{eq:temp70}-\eqref{eq:temp2}, which yields
\begin{equation}
\sum_{j\in\mathcal P(x)} w_{\theta,j}(x)
\le (N-1)e^{-\theta\,\eta(R)},
\quad x\in T_\varepsilon \cap B_R^{\rm c}.
\end{equation}
Using $\sum_{i\in\mathcal A(x)}w_{\theta,i}(x) = 1 - \sum_{j\in\mathcal P(x)}w_{\theta,j}(x)$, we obtain for all $x\in T_\varepsilon \cap B_R^{\rm c}$,
\begin{equation}
L_F \tilde h_\theta(x)
\ge r_\infty 
- (N-1)\big(r_\infty + C(1+R)^p\big)e^{-\theta\,\eta(R)}.
\label{eq:tailbound}
\end{equation}
\Added{The strict condition below ensures strict positivity in the tail:}
\begin{equation}
r_\infty 
> (N-1)\big(r_\infty + C(1+R)^p\big)e^{-\theta\,\eta(R)}
\ \Longleftrightarrow\
\theta > \theta^*_{\mathrm{tail}},
\end{equation}
where $\theta^*_{\mathrm{tail}}
\triangleq \frac{1}{\eta(R)}
\log\!\Bigg(
\frac{(N-1)\big(r_\infty + C(1+R)^p\big)}{r_\infty}
\Bigg)$.

Finally, per \eqref{eq:softmin_bounds}, $\theta \geq \tfrac{\log{N}}{\varepsilon} \implies \partial \tilde \Sc_\theta \subset T_\varepsilon$.
Therefore, choosing $\theta > \theta^*$ satisfies \eqref{eq:maingoal}, where
\begin{equation}
\theta^* 
= \max\!\left\{
\frac{\log N}{\varepsilon}, \theta^*_{\mathrm{core}}, \theta^*_{\mathrm{tail}}
\right\}.
\end{equation}

The strict positiveness of the Lie derivative implies $\nabla \tilde h_\theta(x)\neq 0$ on $\partial\tilde \Sc_\theta$, and thus $0$ is a regular value of $\tilde h_\theta$. Then $\tilde \Sc_\theta$ is {strictly} safe based on Aubin's strict invariance condition \cite{aubin1991viability}; and by Theorem~\ref{theo:converse}, $\tilde h_\theta$ is an eCBF $\forall \theta > \theta^*$.

\end{proof}

\section{Connection to the BCBF Framework}
\label{sec:backup}

We now connect the general results of Section~\ref{sec:main} to the BCBF setting introduced in Section~\ref{subsec:backup}.
Let a (open-loop) system $f\in C^1$ as in \eqref{eq:system}, a set $\Sc$ and a backup set $\Sb$ defined with functions $h,\hb\in C^1$, respectively. 
Suppose that a backup controller $\kb\in C^1$ yields a closed-loop vector field $\Fb(x)=f(x,\kb(x))\in C^1$, where the flow is denoted as $\phi(x_0,\tau)$ with an initial condition $x(0)=x_0 \in \Sc$.
Recall from Section~\ref{subsec:backup}
\begin{equation}
\label{eq:BCBF_hatB}
    \hat \Bc = \{x \in \R^n ~|~ \hat b(x)\ge0\}, \quad \hat b(x) = \min_{i\in \IbN} b_i(x),
\end{equation}
where
\begin{equation}
b_i(x)= 
\begin{cases}
h(\phi(x,\tau_i)), &~{\rm if}~ i \in \IbN \setminus \{N\}, \\
\hb( \phi(x,T) ), &~{\rm if}~ i=N
\end{cases}
\end{equation}
for a finite horizon $T>0$ and a sampling interval $\Delta \tau > 0$ with $\tau_i \in \{0, \Delta \tau, \cdots , T\} $ for all $i \in \IbN$ with $N = T / \Delta \tau + 1$.
We use the soft-min inner-approximation with a $\theta>0$:
\begin{equation}
\label{eq:btilde}
    \tilde b_\theta(x) \triangleq -\frac{1}{\theta}\log\Big(\sum_{i \in \IbN} e^{-\theta b_i(x)}\Big).
\end{equation}
Our goal is to find conditions certifying $\tilde b_\theta$ as a CBF.

We start with defining the backward-expanded set of the (terminal) backup set $\Sb$:
\begin{equation}
    \Bc_N \triangleq \{ x \in \R^n ~|~ b_N(x) \triangleq \hb( \phi(x,T) ) \ge 0 \}.
\end{equation}
Note that $\Bc_N$ consists of points $x$ such that, starting from $x$ at $\tau = 0$, the flow $\phi(x,\tau)$ reaches in $\Sb$ at $\tau = T$. Moreover, $\Bc_N \supseteq \hat \Bc$.
What follows are useful properties of $\Bc_N$.
\begin{lemma}
\label{lem:semigroup}
Let $\Fb,\hb\in C^1$, let $\Sb$ be a compact set, let 0 be a regular value of $\hb$, \Added{assume that the flow $\phi(x,\tau)$ of $\Fb$ is uniquely defined for all $|\tau|\le T$ and all $x \in \R^n$}. 
If $\Fb$ is strictly safe w.r.t. $\Sb$, then it is also strictly safe w.r.t. $\Bc_N$. Moreover, $\Bc_N$ is compact.
\end{lemma}

\begin{proof}
(i)~{\em Compactness.}~
\Added{Define the map $\Phi_\tau(x)=\phi(x,\tau)$. As the solution flow uniquely exists on $\tau \in [-T,T]$ for all $x\in\R^n$, 
$\Phi_{\tau}$ is injective, because no two different initial conditions can lead to a same point in its image (if $\Phi_{\tau}(x_1)=\Phi_{\tau}(x_2)$, then $x_1=x_2$). 
Thus, $\Phi_{\tau}: \R^n \mapsto \Phi_{\tau}(\R^n)$ is bijective for all $|\tau|\le T$, where $\Phi_{\tau}(\R^n)$ is its image.}
Furthermore, since $\Fb \in C^1$, $\Phi_{\tau}$ is $C^1$ for all $|\tau|\le T$. 
Thus, $\Phi_{\tau}$ is a $C^1$ diffeomorphism with a $C^1$ inverse $\Phi_{-\tau}$ for all $|\tau|\le T$ \cite{lee2003smooth}.
This implies that $\Bc_N$ is the continuous image of the compact set $\Sb$ under the inverse: $\Bc_N = \Phi_{-T}(\Sb)$, and therefore $\Bc_N$ is compact.

\smallskip
(ii)~{\em Safety.}~We use the semigroup property of the flow similar to \cite{gurriet2020scalable}---but for strict safety. 
Take any $x \in \Bc_N$, then $\Phi_{T}(x) \in \Sb$ by definition of $\Bc_N$. Since the backup system is strictly safe w.r.t. $\Sb$, for every $s > 0$ we have $\Phi_{T+s}(x)\in {\rm Int}(\Sb)$. Thus, for every $s>0$, $b_N( \Phi_{s}(x) ) = \hb( \Phi_{T+s}(x) ) >0$, so $\Phi_{s}(x) \in {\rm Int}(\Bc_N)$. This is exactly the safety definition in Definition~\ref{def:safety}.
\end{proof}

Safety of $\Bc_N$ creates a trivial case. In particular, if $\Bc_N \subset {\rm Int}(\Sc)$, then backup trajectories $\phi$ cannot visit $\partial \Sc$ as they have to stay inside $\Bc_N$, and this implies $\hat \Bc = \Bc_N$. In this case, Lemma~\ref{lem:semigroup} provides us with all that is required to invoke Theorem~\ref{theo:compact}; so $\hat b(x) = b_N(x)$ can be certified as a CBF (without using $h$).

In the more general case of $\hat \Bc \subset \Bc_N$, we need to pay close attention to the boundary points of $\hat \Bc$ where discretized `slice'\footnote{\Added{By `slice' we mean evaluating $h$ at a fixed time slice of the backup trajectory: the constraint $b_i(x) = h(\phi(x,\tau_i))$ records the value of $h$ after flowing the system from $x$ for $\tau_i$ seconds of time.}}
constraints of $b_i(x) = h( \phi(x,\tau_i))$ are active for some $i < N$.
Observe that $\phi(x,\tau_i) \in \partial \Sc $ at some $x \in \partial \hat \Bc$ if a slice becomes active $b_i(x) = \hat b(x)$ for some $i < N$.
Furthermore, since $\Bc_N$ is {strictly} safe under the backup flow, all such evaluations necessarily lie in $\Bc_N$ as well.
Consequently, every potential activation of some $i < N$ on the boundary of $\hat \Bc$ occurs on the \emph{backup-reachable boundary of $\Sc$}:
\begin{equation}
    \label{eq:Scrit}
    \Cc \triangleq \partial \Sc \cap \Bc_N.
\end{equation}
\begin{assumption}
\label{ass:BCBF}
    The following holds:
    \begin{equation}
    \label{eq:BCBFassumption_onC}
        L_{\Fb} h(y) > 0, ~~\forall y \in \Cc.
    \end{equation}
\end{assumption}

Assumption~\ref{ass:BCBF} translates as the strict safety wherever $\partial \Sc$ becomes relevant for the backup trajectory. 
Together with the strict safety of $\Sb$, this yields the strict safety of $\hat \Bc$, which is required to invoke Theorem~\ref{theo:compact}.
Note that the condition is required only on $\Cc$ because this set consists of the boundary points of the target safe set $\Sc$ and the backward-expanded set of $\Bc_N$, i.e., exactly the states at which the backup-based construction is applied. Any other boundary point of $\Sc$ that is outside $\Cc$ is not relevant for the formulation as $x \notin \Bc_N$ implies $\phi(x,T) \notin \Sb$.
{\color{black} We note that Assumption~\ref{ass:BCBF}, together with the other hypotheses of Corollary~1, can be verified directly from the system data, either analytically or numerically; the example in Section~V illustrates such a verification on a grid.}


\begin{corollary}
\label{cor:softmin_bcbf}
Let $\Fb,h,\hb \in C^1$, let $\Sb$ be compact, 0 is a regular value of $\hb$, and suppose the flow $\phi(x,\tau)$ of $\Fb$ exists for all $|\tau|\leq T$ for every initial condition $x \in \hat \Bc$. 
If $\Fb$ is {strictly} safe w.r.t. $\Sb$ and Assumption~\ref{ass:BCBF} holds,
then there exists a $\theta^*>0$ such that $\tilde b_\theta$ is a CBF for \eqref{eq:system} for all $\theta\ge\theta^*$.
\end{corollary}

\begin{proof}
We verify the four hypotheses of Theorem~\ref{theo:compact} (see Lemma~\ref{lem:compact}) in order.

\smallskip
\paragraph{$\hat \Bc$ is compact}
By Lemma~\ref{lem:semigroup}, $\Bc_N$ is compact and {strictly} safe under $\Fb$. 
For any $i<N$, we have $b_i(x)=h(\phi(x,\tau_i))$, which is continuous, hence $\Bc_i=\{x~|~b_i(x)\ge 0\}$ is closed.
Therefore $\hat \Bc=\bigcap_{i=1}^N B_i\subseteq B_N$ is the intersection of a compact set with finitely many closed sets, and is thus compact. 

\smallskip
\paragraph{$b_i\in C^1$ for all $i\in \IbN $}
Since $\Fb,h,h_b\in C^1$ and the flow exists for all $|\tau| \le T$, the flow $x\mapsto\phi(x,\tau)$ is $C^1$ for each fixed $\tau$.
Thus, subconstraint functions $b_i(x)=h(\phi(x,\tau_i))$ for a $i<N$ and $b_N(x)=\hb(\phi(x,T))$ are all $C^1$.
Indeed
\begin{equation}
\label{eq:corol1}
\nabla b_i(x)=
\begin{cases}
    D_x^\top\phi(x,\tau_i)\,\nabla h(\phi(x,\tau_i)), & i<N,\\
    D_x^\top\phi(x,T)\,\nabla \hb(\phi(x,T)), & i=N,
\end{cases}
\end{equation}
where $D^\top_x\phi \!\in\! \R^{ n \times n}$ is the \Added{transpose} of the Jacobian matrix of $\phi(x,\cdot)$ w.r.t. $x$.

\smallskip
\paragraph{MFCQ holds on $\partial \hat \Bc$}
For an arbitrary $x\in\partial \hat \Bc$, we consider two cases: \Added{$i \in \Ac(x)$ for a $i < N$, or $N \in \Ac(x)$.
The former implies} $\phi(x,\tau_i)\in\partial \Sc$. In addition, since $\hat \Bc\subseteq \Bc_N$, and $\Bc_N$ is strictly safe, then $\phi(x,\tau_i)\in \Bc_N$.
Then Assumption~\ref{ass:BCBF} applies, and we have that $L_{\Fb} b_i(x) > 0$.
Thus, taking $v=\Fb(x)$ satisfies \eqref{eq:MFCQ} for all $x\in \partial \hat \Bc$ \Added{when $i$-th constraint is active with $i<N$.
If $N \in \Ac(x)$ for a $x\in \partial \hat \Bc$}, then $\phi(x,T)\in\partial \Sb$, and strict safety of $\Sb$ yields $L_{\Fb}\hb(\phi(x,T))>0$ as 0 is a regular value of $\hb$.
Similarly, $v=F_b(x)$ satisfies \eqref{eq:MFCQ} for this case. Thus, MFCQ holds on $\partial \hat \Bc$ with $v=\Fb(x)$.

\paragraph{(The backup system is strictly safe w.r.t. $\hat \Bc$}
Under MFCQ, Aubin’s strict invariance condition \cite{aubin1991viability} reduces to the strict positiveness of the Lie derivative on boundary, cf. \eqref{eq:safetyassumption}. 
We verify this inequality at an arbitrary $x\in\partial\hat \Bc$ \Added{for the same two cases as in the previous part: $i \in \Ac(x)$ for a $i < N$, or $N \in \Ac(x)$}.
\emph{Case $i<N$,} as discussed in the previous part,
Assumption~\ref{ass:BCBF} applies, and $L_{\Fb} b_i(x) > 0$ for all  $x\in\partial \hat \Bc$ \Added{when $i$-th constraint is active with $i<N$}.
\emph{Case $i=N$,} means $\phi(x,T)\in\partial \Sb$ and, since the backup closed loop is strictly safe w.r.t. $\Sb$ and 0 is a regular value for $\hb$, we have $L_{\Fb} \hb(\phi(x,T))>0$. 
Thus \eqref{eq:safetyassumption} holds, and strict safety follows from Aubin’s strict invariance condition \cite{aubin1991viability}.

\smallskip
Having established (a)--(d), all hypotheses of Theorem~\ref{theo:compact} are satisfied, and this concludes the proof.
\end{proof}

{\color{black}
We note that conditions required by Corollary~\ref{cor:softmin_bcbf} can be verified directly from the system data, either analytically or numerically; the example in Section~V illustrates such a verification on a grid.
Beyond certifying $\tilde b_\theta$ as a CBF, Corollary~\ref{cor:softmin_bcbf} brings the practical recipe of Section~\ref{sec:compact} into the BCBF setting: once $\theta > \theta^*$ is fixed, a lower bound $\gamma^*$ can be computed from~\eqref{eq:gamma_star} for $\alpha(r) = \gamma r$. 
In the next section, we illustrate this procedure on a numerical example.
}

\begin{remark}
\label{rem:only-Sb-compact}
Corollary~\ref{cor:softmin_bcbf} relies on the compactness of the backup set $\Sb$; and we can still certify $\tilde b_\theta$ as a CBF even in the case $h$ defines an unbounded $\Sc$ (provided that Assumption~\ref{ass:BCBF} holds).
This relaxes the stronger (global) compactness assumptions found in the literature; e.g., \cite{gurriet2020scalable}.
If the backup set $\Sb$ is not compact, then we can still use Theorem~\ref{theo:unbounded} to certify $h$ as an eCBF, which would imply the feasibility of the safety constraint in~\eqref{eq:safetyfilter} all the same. We leave this as a topic for future work.
\end{remark}

\section{Example}
\label{sec:example}

\begin{figure}[t]
    \centering
\includegraphics[width=0.75\columnwidth]{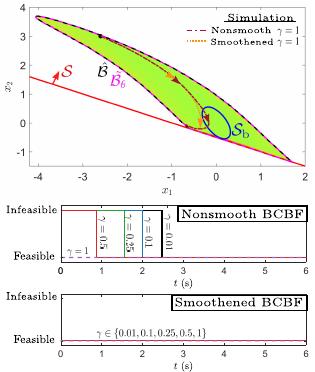}
    \caption{\color{black} (Top) Sets $\Sc$, $\Sb$, $\hat \Bc$ and $\tilde \Bc_\theta$ for the example, and simulation trajectories with safety filter \eqref{eq:safetyfilter} with $\hat b$ (brown) and $\tilde b_\theta$ (orange) for $\alpha(r) = \gamma r$. (Middle) Feasibility of the safety constraint over time for $\hat b$ with $\gamma \in \{0.01, 0.1, 0.25, 0.5, 1\}$: all $\gamma < 1$ exhibit intervals of infeasibility. (Bottom) Feasibility of the safety constraint over time for $\tilde b_\theta$: every run is feasible, consistent with the a priori certificate $\gamma^* > 0$. }
    \label{fig:figure}
    \vspace{-4 mm}
\end{figure}

{\color{black}
We illustrate the method on a double integrator
\begin{equation}
\label{eq:ex_system}
\dot x_1 = x_2, \quad \dot x_2 = u, \quad |u| \le u_{\max}  = 2,
\end{equation}
and the safe set
\begin{equation}
\label{eq:ex_S}
\Sc = \{x \in \R^2 ~|~ h(x) \ge 0\}, \quad h(x) = x_1 + 2 x_2 + 1.
\end{equation}
Note that $\Sc$ is unbounded, and $h$ is not a CBF for~\eqref{eq:ex_system} due to the input bound.
This motivates the BCBF.}

\noindent \textbf{Backup controller and backup set.}
{\color{black} 
We use a linear backup controller $\kb(x) = -x_1 - 2 x_2$, and to obtain a $C^1$ backup closed-loop, we saturate $\kb$ smoothly:}
\begin{equation}
\label{eq:ex_Fb}
\Fb(x) = \begin{bmatrix} x_2 \\ u_{\max}\tanh\!\bigl((-x_1 - 2 x_2)/u_{\max}\bigr) \end{bmatrix}.
\end{equation}
{\color{black} Solving the Lyapunov equation with $\Fb$ linearized around the origin and $Q=-I$ yields $P = \begin{bmatrix} 3/2 & 1/2 \\ 1/2 & 1/2 \end{bmatrix}$. We define the (compact) backup set as the Lyapunov sublevel set
\begin{equation}
\Sb = \{x \in \R^2 ~|~ \hb(x) \ge 0\}, \quad \hb(x) = \rho - x^\top P x,
\end{equation}
with $\rho = 0.1$. This choice guarantees $\Sb \subset \Sc$, $|\kb(x)| < u_{\max}$ on $\Sb$, and ensures that $\Fb$ is strictly safe w.r.t. $\Sb$.
}

\noindent \textbf{BCBF construction.}
{\color{black} 
With horizon $T = 2$ and discretization step $\Delta \tau = 0.1$, we construct the functions $\{b_i\}_{i=1}^{N}$ as in~\eqref{eq:BCFB_bi} with $N = 21$, and define $\hat b$, $\hat \Bc$ as in~\eqref{eq:BCBF_hatB}. 
The safety filter \eqref{eq:safetyfilter} with the nonsmooth $\hat b$ enforces the $N$ simultaneous constraints~\eqref{eq:BCBF_condition}. Existing nonsmooth BCBF formulations do not provide a constructive procedure for selecting the class-$\mathcal{K}$ functions $\alpha_i$, leaving the user to tune them by trial and error until a configuration yielding a feasible safety constraint is found.
To illustrate the consequences of this ad hoc selection, we simulate the safety filter ($k_{\rm des} = - u_{\rm max}$) with $\alpha(r) = \gamma r$ for $\gamma \in \{0.01, 0.1, 0.25, 0.5, 1\}$ from the initial condition $x_0 = (-2.5, 3)$; the feasibility of safety constraints in each run is reported over time in the middle panel of Fig.~\ref{fig:figure}. All choices with $\gamma < 1$ yield intervals of infeasibility (the trajectory for $\gamma = 1$ is shown in the top panel). }

\noindent \textbf{Smoothed BCBF.}
{\color{black} 
We replace the pointwise minimum with the soft-min $\tilde b_\theta$ as in~\eqref{eq:btilde}. Verifying the hypotheses of Corollary~\ref{cor:softmin_bcbf}: $\Sb$ is compact, $0$ is a regular value of $\hb$, $\Fb \in C^1$, $\Fb$ is strictly safe w.r.t. $\Sb$. 
Assumption~\ref{ass:BCBF} is checked numerically on a grid: at boundary points $x \in \partial \hat \Bc$ where an active constraint $i<N$ corresponds to $\phi(x,\tau_i) \in \Cc$, we evaluate $L_{\Fb} b_i(x)$ and find it strictly positive throughout, confirming $L_{\Fb} h(y) > 0$ on the sampled portion of $\Cc$. The case $i = N$ follows directly from the strict safety of $\Sb$, which yields $L_{\Fb} \hb(\phi(x,T)) > 0$ at such points.
}

{\color{black} 
We set $\varepsilon = \rho$ so that $T_\varepsilon = \hat \Bc$. The quantities $M, r, d$ from Lemma~\ref{lem:compact} and $L_{\rm max}$ from \eqref{eq:gamma_star} are evaluated numerically on a uniform grid over a box containing $\hat \Bc$.
Trajectories and sensitivities needed for $b_i$ and gradients $\nabla b_i$ are obtained by forward-integrating $\Fb$ together with the variational equation $\dot S = (\partial \Fb / \partial x)\, S$, $S(0) = I$, using a fixed-step Runge--Kutta scheme\footnote{Related code are in \href{github.com/anilalan-umich/smooth-bcbf}{https://github.com/anilalan-umich/smooth-bcbf}}.
This yields $M \approx 2.01$, $r \approx 0.027$, and $d \approx 0.005$, which via~\eqref{eq:thetastar_compact} give $\theta^* \approx 1451$; we set $\theta = 1524$, which gives a very close approximation of $\hat \Bc$ with $\tilde \Bc_\theta$. 
The resulting $L_{\rm max}$ is numerically zero as the backup controller can enforce a strict increase of $\tilde{b}_\theta$ everywhere on $\tilde \Bc_\theta$.
This means that \eqref{eq:gamma_star} is satisfied with any $\gamma > 0$. Crucially, this certifies that \emph{any linear class-$\mathcal{K}$ function renders the safety constraint in~\eqref{eq:safetyfilter} feasible on $\tilde \Bc_\theta$}, a conclusion that cannot be provided for the classical nonsmooth BCBF formulation a priori.
This is demonstrated in the bottom panel of Fig.~\ref{fig:figure}, where the safety filter \eqref{eq:safetyfilter} with $\tilde b_\theta$ is simulated from the same initial condition $x_0$ for each $\gamma \in \{0.01, 0.1, 0.25, 0.5, 1\}$ and the constraint remains feasible throughout in every run; the closed-loop trajectory for $\gamma = 1$ is shown in the top panel.
}

\section{Conclusion}
\label{sec:conclusion}
We have established uniform feasibility guarantees for CBF-based controllers such as safety filters using converse safety theorems when the safe set is constructed using smoothed minima of $C^1$ functions (via soft-min). For compact safe sets, an explicit parameter is designated for {certifying the soft-min function as a CBF}, which implies the feasibility of the CBF-based safety constraint. For noncompact sets, a set of tail conditions was shown to yield an extended CBF (a more relaxed formulation of CBF) with the same feasibility results. Projecting theoretical findings onto the BCBFs, we show a set of sufficient conditions for feasibility in BCBF formulation. 
Future work will focus on relaxing assumptions of compact backup set on the backup-reachable boundary of the safe set. 
{\color{black} Investigating the time-varying extension is another direction for future work.}

\bibliographystyle{IEEEtran}
\bibliography{_Aux/alan}

\end{document}